\documentclass[
11pt]{amsart}

\usepackage{amsmath}
\usepackage{amssymb}
\usepackage{latexsym}
\usepackage{amsfonts}
\usepackage{graphpap}

\begin{document}

\newcommand{\ci}[1]{_{ {}_{\scriptstyle #1}}}
\newcommand{\ti}[1]{_{\scriptstyle \text{\rm #1}}}

\newcommand{\norm}[1]{\ensuremath{\|#1\|}}
\newcommand{\abs}[1]{\ensuremath{\vert#1\vert}}
\newcommand{\nm}{\,\rule[-.6ex]{.13em}{2.3ex}\,}

\newcommand{\lnm}{\left\bracevert}
\newcommand{\rnm}{\right\bracevert}

\newcommand{\p}{\ensuremath{\partial}}
\newcommand{\pr}{\mathcal{P}}

\newcounter{vremennyj}

\newcommand\cond[1]{\setcounter{vremennyj}{\theenumi}\setcounter{enumi}{#1}\labelenumi\setcounter{enumi}{\thevremennyj}}

\newcommand{\pbar}{\ensuremath{\bar{\partial}}}
\newcommand{\db}{\overline\partial}
\newcommand{\D}{\mathbb{D}}
\newcommand{\T}{\mathbb{T}}
\newcommand{\C}{\mathbb{C}}
\newcommand{\N}{\mathbb{N}}
\newcommand{\bP}{\mathbb{P}}

\newcommand{\bS}{\mathbf{S}}
\newcommand{\bk}{\mathbf{k}}

\newcommand\cE{\mathcal{E}}
\newcommand\cP{\mathcal{P}}
\newcommand\cC{\mathcal{C}}
\newcommand\cH{\mathcal{H}}
\newcommand\cU{\mathcal{U}}
\newcommand\cQ{\mathcal{Q}}

\newcommand{\be}{\mathbf{e}}

\newcommand{\la}{\lambda}
\newcommand{\e}{\varepsilon}

\newcommand{\td}{\widetilde\Delta}

\newcommand{\tto}{\!\!\to\!}
\newcommand{\wt}{\widetilde}
\newcommand{\shto}{\raisebox{.3ex}{$\scriptscriptstyle\rightarrow$}\!}

\newcommand{\La}{\langle }
\newcommand{\Ra}{\rangle }
\newcommand{\ran}{\operatorname{ran}}
\newcommand{\tr}{\operatorname{tr}}
\newcommand{\codim}{\operatorname{codim}}
\newcommand\clos{\operatorname{clos}}
\newcommand{\spn}{\operatorname{span}}
\newcommand{\lin}{\operatorname{Lin}}
\newcommand{\rank}{\operatorname{rank}}
\newcommand{\re}{\operatorname{Re}}
\newcommand{\vf}{\varphi}
\newcommand{\f}{\varphi}


\newcommand{\entrylabel}[1]{\mbox{#1}\hfill}

\newenvironment{entry}
{\begin{list}{X}%
  {\renewcommand{\makelabel}{\entrylabel}%
      \setlength{\labelwidth}{55pt}%
      \setlength{\leftmargin}{\labelwidth}
      \addtolength{\leftmargin}{\labelsep}%
   }%
}%
{\end{list}}



\numberwithin{equation}{section}

\newtheorem{thm}{Theorem}[section]
\newtheorem{lm}[thm]{Lemma}
\newtheorem{cor}[thm]{Corollary}
\newtheorem{prop}[thm]{Proposition}

\theoremstyle{remark}
\newtheorem{rem}[thm]{Remark}
\newtheorem*{rem*}{Remark}

\title[Similarity]{Similarity of Cowen-Douglas operators to the backward Dirichlet shift}
\author{Hyun-Kyoung Kwon}
\thanks{The work of H.~Kwon was supported by the CARSCA grant of the University of Alabama.}
\address{Department of Mathematics\\ The University of Alabama\\ Tuscaloosa, AL USA 35487}


\begin{abstract}
We show that the same similarity characterization obtained for Cowen-Douglas operators to the backward shift operators on reproducing kernel Hilbert spaces with analytic kernels can be used to describe similarity in the Dirichlet space setting. As in previous proofs, a model theorem that allows one to get the eigenvector bundle structure of the operator plays a crucial role.
\end{abstract}

\maketitle
\setcounter{tocdepth}{1}
\tableofcontents
\section*{Notation}
\begin{entry}
\item[$:=$] equal by definition;\medskip
\item[$\C$] the complex plane;\medskip
\item[$\D$] the unit disk, $\D:=\{z\in\C:\abs{z}<1\}$;\medskip
\item[$\T$] the unit circle, $\T:=\p\D=\{z\in\C:\abs{z}=1\}$;\medskip
\item[$\frac{\p}{\p z}, \frac{\p}{\p \overline z}$] $\p$ and $\db$ derivatives,
$\frac{\p}{\p z}  := (\frac{\p}{\p x} - i \frac{\p}{\p y})/2$, $\frac{\p}{\p \overline z}
 := (\frac{\p}{\p x} + i \frac{\p}{\p y})/2$; \medskip
\item[$\Delta$]normalized Laplacian, $\Delta = \db \p = \p
\db = \frac14\left(\frac{\p^2}{\p x^2} + \frac{\p^2}{\p y^2} \right)$;\medskip
\item[$\mathfrak{S}_2$]Hilbert-Schmidt  class of operators;\medskip
\item[$\norm{\cdot}, \nm\cdot \nm$]  norms, since we are dealing with matrix- and operator-valued functions, we will use the symbol $\|\,.\,\|$ (usually with a
subscript) for the norm in a function space, while $\nm\,.\,\nm$
is used for the norm in the underlying vector (operator) space.
 Thus, for a vector-valued function $f$, the symbol $\|f\|_2$ denotes its $L^2$-norm, but the symbol $\nm f\nm$ stands
for the scalar-valued function whose
value at a point $z$ is the norm of the vector $f(z)$;  \medskip

\item[$H^\infty$] the Hardy space of functions bounded and analytic on $\mathbb{D}$; \medskip

\item[$H^\infty_{\! \mathcal{E_*}\shto \mathcal{E}}$] the operator Hardy class of bounded analytic
 functions whose values are bounded operators from a Hilbert space $\mathcal{E}_*$ to another Hilbert space $\mathcal{E}$,
$$
\|F\|_\infty := \sup_{z\in \D}
\nm F(z)\nm=\underset{\xi\in
\T}{\operatorname{esssup}}\nm F(\xi)\nm,$$
where $\hat{F}(\xi)=\lim_{r \rightarrow 1^{-}} F(r\xi);$ and \medskip

\item[$T_\Phi$] the Toeplitz operator with symbol $\Phi$ (possibly operator-valued). \medskip

\end{entry}

\setcounter{section}{-1}

\section{Introduction}
In previous work, partial results to the problem of the similarity of Cowen-Douglas operators in terms of curvatures of the corresponding eigenvector bundles were obtained when one restricts attention to the similarity to the backward shift operator on the Hardy space or the weighted Bergman spaces \cite{DouglasKwonTreil}, \cite{KwonTreil}. Unlike the case for unitary equivalence that was completely solved by M. J. Cowen and R. G. Douglas in \cite{CowenDouglas}, a characterization for similarity is still far from being complete and the starting point for research in this direction comes from the backward shift operator defined on holomorphic function spaces that are of interest to operator theorists, the Hardy and the Bergman spaces- recall that the backward shift operator is not only simple to work with but also serves as an important model \cite{Agler1}, \cite{Agler2}, \cite{NagyFoias1}. In the present paper, we focus on the backward shift operator on the remaining holomorphic function space that attracts much attention, the Dirichlet space.

As in previous work, we first get the eigenvector structure of the operators involved and then as noted by B. Sz-Nagy and C. Foias in \cite{NagyFoias1}, \cite{NagyFoias2}, and \cite{NagyFoias3}, use the relationship between the similarity and corona problems to obtain a similarity characterization involving the curvatures of the bundles. It will be shown that the same characterization that holds in the Hardy and Bergman spaces is true in the Dirichlet space setting.

\section{Preliminaries}
The Dirichlet space $\mathcal{D}$ consists of all analytic functions $f(z)=\sum_{n=0}^{\infty} \hat{f}(n) z^n$ defined on the unit disk $\mathbb{D}$ of the complex plane satisfying
$$\|f\|^2=\sum_{n=0}^{\infty} (n+1)|\hat{f}(n)|^2
 < \infty.
$$
Just like the Hardy and weighted Bergman spaces, the Dirichlet space $\mathcal{D}$ is a reproducing kernel Hilbert space with reproducing kernel
$$
k_{\lambda}(z)= \frac{1}{\bar{\lambda}z}\log \frac{1}{1-\bar{\lambda}z}=1+\frac{1}{2}{\bar{\la}z}+\frac{1}{3}{\bar{\la}}^2z^2+\cdots.
$$
We can define $\mathcal{D}_{\mathcal{E}}$, the vector-valued analogue of $\mathcal{D}$ taking values in a Hilbert space, in an obvious way. We will write $\mathcal{D}_n$ when $\mathcal{E}$ is of dimension $n$.

The operator of multiplication by $z$ on $\mathcal{D}$, denoted $D$, is a bounded linear operator. We will denote its adjoint by $D^*$ and call it the backward shift
operator. We can define $D^*_{\mathcal{E}}$ is an analogous way and by taking note that since an orthonormal basis for $\mathcal{D}$ is given by $\{\frac{z^n}{\sqrt{n+1}}\}_{n=0}^{\infty}$, conclude that $D^*_{\mathcal{E}}$ can be viewed as the weighted backward shift
$$
S_{\alpha}(x_0,x_1,x_2, \cdots)=({\frac{\sqrt 2}{\sqrt 1}} x_1, {\frac{\sqrt 3}{\sqrt 2}} x_2, {\frac{\sqrt 4}{\sqrt 3}} x_3, \cdots),
$$
corresponding to the weight sequence ${\alpha}=({\alpha}_n)_{n=1}^{\infty}=(\frac{\sqrt{n+1}}{\sqrt n})_{n=1}^{\infty}$. Moreover,
$$
D^*k_{\bar{\lambda}}=\lambda k_{\bar{\lambda}},
$$
for all $\lambda \in \mathbb{D}$, so that the eigenvectors of $D^*_{\mathcal{E}}$ corresponding to the eigenvalue $\lambda \in \mathbb{D}$ are $k_{\bar{\lambda}}e$ for $e \in \mathcal{E}$.

\section{Main Results}
Let $\mathcal{H}$ be a separable Hilbert space. We assume the following about the operator $T \in \mathcal{L}(\mathcal{H})$ in consideration:
\begin{enumerate}

\item
$\sum_{n=1}^{\infty} \frac{\|T^n\|^2}{n+1} \leq 1$;

\item
$\bigvee_{\lambda \in \mathbb{D}} \ker (T-\lambda)=\mathcal{H}$ ; and

\item the subspaces $\ker(T-\lambda)$   depend analytically on the spectral parameter
$\la \in \D$.\
\end{enumerate}

\medskip
By assumption \cond3, we have for each $\lambda \in\D$, a neighborhood $U_{\lambda}$ of $\lambda$ and an operator-valued analytic function $F_{\lambda} \in H^{\infty}_{\mathcal{H} \rightarrow \mathcal{H}}$ defined on $U_{\lambda}$ with $\ran F_{\lambda}(\omega)=\ker(T-\omega)$ for $\omega \in U_{\lambda}$. This analytic function $F_\lambda$ has a left inverse in $L^{\infty}_{\mathcal{H} \rightarrow \mathcal{H}}$ and it can be easily shown that $\dim \ker(T-\lambda)$ is constant for all $\lambda \in \mathbb{D}$. Note that a Cowen-Douglas operator in $B_m(\mathbb{D})$ for $m$ a positive integer, will satisfy assumptions \cond2 and \cond 3 (cf. \cite{CowenDouglas}). The disjoint union
$$
E_T=\coprod_{\lambda \in \mathbb{D}} \ker(T-\lambda)=\{(\lambda, v_{\lambda}): \lambda \in \mathbb{D}, v_{\lambda} \in \ker (T-\lambda)\}
$$
is a Hermitian holomorphic vector bundle over $\mathbb{D}$ with the metric inherited from $\mathcal{H}$ and the natural projection $\pi$, $\pi(\lambda, v_{\lambda})=\lambda$.

In order to state the main result of the paper, we define a $\mathcal{C}^{\infty}$ function $\Pi$ defined on $\mathbb{D}$, where
$$
\Pi(\lambda)=\ker(T-\lambda),
$$
for $\lambda \in  \mathbb{D}$.
The following theorem is the main result of the paper:
\begin{thm} \label{t0.1} Let $T \in \mathcal{L}(\mathcal{H})$ satisfy assumptions (1) through (3). If
$$\dim\ker(T-\lambda)=n<\infty$$ for every $\lambda\in\D$, and $\Pi(\lambda)$ denotes the orthogonal projection onto
$\ker(T-\lambda)$, then the following statements are equivalent:

\begin{enumerate}
\item
T is similar to the backward shift operator $D^*_n$ on the vector-valued space $\mathcal{D}_n$, i.e., there exists an invertible operator $A: \mathcal{D}_n \rightarrow \mathcal{H}$ satisfying $TA=AD^*_n$;

\item
The eigenvector bundles of $T$ and $D^*_n$ are ``uniformly equivalent''. i.e., there exists a holomorphic bundle map bijection $\Psi$ from the eigenvector bundle of
$D^*_n$ to that of $T$ satisfying
$$
\frac{1}{c} \| v_{\lambda}  \|_{\mathcal{D}_n} \leq \| \Psi ( v_{\lambda} )\|
_{\mathcal{H}} \leq c \| v_{\lambda} \| _{\mathcal{D}_n}
$$
for some constant $c>0$ and for all $v_{\lambda} \in \ker(D^*_n -\lambda)$ and all $\la\in \D$;

\item
There exists a bounded subharmonic function $\vf$ defined on $\mathbb{D}$ such that
$$
\Delta \vf (\lambda) \geq \left \bracevert \frac{\partial\Pi(\lambda)}{\partial
\lambda}\right \bracevert ^2_{\mathfrak{S}_2}+\frac{n[\log(1-|\lambda|^2)+|\lambda|^2]}{[\log(1-|\lambda|^2)(1-|\lambda|^2)]^2}
\quad \text{ for all } \lambda \in \D.
$$

\end{enumerate}

\end{thm}

Statement (3) of the theorem is equivalent to the existence of a bounded function $\vf$ defined on $\mathbb{D}$ with
\begin{equation}
\Delta \vf (\lambda) = \left \bracevert \frac{\partial\Pi(\lambda)}{\partial
\lambda}\right \bracevert ^2_{\mathfrak{S}_2}+\frac{n[\log(1-|\lambda|^2)+|\lambda|^2]}{[\log(1-|\lambda|^2)(1-|\lambda|^2)]^2}
\end{equation}
for all $\lambda \in \D$. These two claims amount to showing that the corresponding Green potential
$$
\mathcal G(\la) := \frac2\pi\iint_\D \log \left|
\frac{z-\la}{1-\overline \la z}\right|  (\left \bracevert
\frac{\partial\Pi(\la)}{\partial \la} \right \bracevert
^2_{\mathfrak{S}_2}+\frac{n[\log(1-|\lambda|^2)+|\lambda|^2]}{[\log(1-|\lambda|^2)(1-|\lambda|^2)]^2 }) dA(z)
$$
is uniformly bounded inside $\D$ ($dA$ is the area measure on $\mathbb{D}$). That statement (1) of the theorem implies statement (2) is obvious and in order to prove the theorem, we will show the implications $(2) \rightarrow (3)$ and $(3) \rightarrow (1)$.

\section{Computing the eigevector bundle of $T$}
Let of first consider the following model theorem by V. Muller \cite{Muller} that plays the role of the theorems of J. Agler \cite{Agler1}, \cite{Agler2}. Recall that $D^*$ is of the form
$$S_{\alpha}(x_1, x_2, \cdots )=({\frac{\sqrt 2}{\sqrt 1}} x_1, {\frac{\sqrt 3}{\sqrt 2}} x_2, {\frac{\sqrt 4}{\sqrt 3}} x_3, \cdots),$$
with the corresponding weight sequence being ${\alpha}=({\alpha}_n)_{n=1}^{\infty}=(\frac{\sqrt{n+1}}{\sqrt n})_{n=1}^{\infty}$.

\begin{thm}
Let ${\alpha}=({\alpha}_n)_{n=1}^{\infty}$ be such that $\alpha_j \geq \alpha_{j+1} > 0$ for all $j \geq 1$. For $T \in \mathcal{H}$, there exists a $S_{\alpha}-$ invariant subspace $\mathcal{K}$ such that $T$ is unitarily equivalent to $S_{\alpha}|\mathcal{K}$ if and only if
$\sum_{n=1}^{\infty}\|T\|^2 b_n \leq 1$, where for $n \geq 1$, $b_n=\alpha^{-2}_n \cdots \alpha^{-2}_1$.
\end{thm}

By assumption (1) placed on our operator $T$, $b_n=\frac{1}{n+1}$ and thus $\alpha_n=\frac{\sqrt {n+1}}{\sqrt n}$. We can hence invoke the above theorem to get that $T$ is unitarily equivalent to $D^*\mathcal{E}$ restricted to some invariant subspace of $\mathcal{D}_{\mathcal{E}}$. Note that the condition $\sum_{n=1}^{\infty}\|T\|^2 b_n \leq 1$ enables us to use the above theorem just like the n-hypercontraction assumption in the model theorem due to J. Agler.

We can therefore conclude that the eigenspaces of $T=D^*\big| \mathcal{K}$ are given by
$$
\ker(T-\la) = \{k_{\overline\la} e: e\in \mathcal{E}(\la)\},
$$
where $\mathcal{E}(\la)$ is the subspace defined by
$$
\mathcal{E}(\lambda)=\{e \in \mathcal{E}; k_{\bar{\la}}e \in \mathcal{K}\}.
$$
Assumption \cond3 of $\ker(T-\la)$ being a holomorphic vector bundle implies that that the family of subspaces $E(\la)$ is again a holomorphic vector bundle over $\mathbb{D}$.

The vector-valued Dirichlet space $\mathcal{D}_{\mathcal{E}}$ can be realized as the tensor product $\mathcal{D} \otimes \mathcal{E}$ and we can then express the eigenvector bundle of $T$ as
$$
\ker(T-\la) = \bigvee_{\lambda \in \mathbb{D}} \{k_{\overline{\la}}\} \otimes \mathcal{E}(\la).
$$
This tensor product form of the eigenvector bundle allows us to represent $\Pi(\la)$, the orthogonal projection onto $\ker(T-\la)$ as a tensor product of the operators $\Pi_(\la)$ and $\Pi_2(\la)$, which are orthogonal projections from $\mathcal{D}$ onto $\bigvee_{\la \in \mathbb{D}} \{k_{\overline{\la}}\}$ and from $\mathcal{E}$ onto $\mathcal{E}(\la)$, respectively:
\begin{equation}
\Pi(\la)=\Pi_1(\la)\otimes\Pi_2(\la).
\end{equation}

We will now see from the theorem stated below that the proof of the main theorem really depends on the second part $\Pi_2(\lambda)$ of $\Pi(\lambda)$. We first note that $\rank \Pi(\lambda)=\rank \Pi_2(\lambda)=n$ and that $\rank \Pi_1(\lambda)=1$.
\begin{thm}
\begin{align*}
{\left \bracevert \frac{\partial\Pi(\lambda)}{\partial\lambda}\right \bracevert}^2_{\mathfrak{S}_2} & =
{\left \bracevert \frac{\partial\Pi_1(\lambda)}{\partial\lambda}\right \bracevert}^2_{\mathfrak{S}_2} +
{\left \bracevert \frac{\partial\Pi_2(\lambda)}{\partial\lambda}\right \bracevert}^2_{\mathfrak{S}_2}
\\
& =
-\frac{n[\log(1-|\lambda|^2)+|\lambda|^2]}{[\log(1-|\lambda|^2)(1-|\lambda|^2)]^2}\, +
{\left \bracevert \frac{\partial\Pi_2(\lambda)}{\partial\lambda}\right \bracevert}^2_{\mathfrak{S}_2}.
\end{align*}
\end{thm}

\begin{proof}
For the first equality, we use the product rule, the fact that for an orthogonal projection $P$, we have
$\left \bracevert P \right \bracevert ^2_ {\mathfrak{S}_2}=\rank P$, and the identities
$$
\Pi_2(\la) \frac{\p\Pi_2(\la)}{\p \la}=0,
$$
and
$$
(I-\Pi_2(\la))\frac{\p\Pi_2(\la)}{\p \la} \Pi_2(\la) = \frac{\p\Pi_2(\la)}{\p \la},
$$
that follow from assumption (3). For details on these identities, we refer the reader to \cite{DouglasKwonTreil} or \cite{KwonTreil}.
To complete the proof, it suffices to show the identity
$$
\left \bracevert \frac{\p \Pi_1(\la)}{\p \la} \right \bracevert ^2 _{\mathfrak{S}_2}= -\frac{(\log(1-|\lambda|^2)+|\lambda|^2}{[\log(1-|\lambda|^2)(1-|\lambda|^2)]^2}.
$$
A simple method for getting this identity is to note that $-\frac{\p \Pi_1(\la)}{\p \la}$ is the second fundamental form \cite{GriffithsHarris}. Then the curvature of $\ker(T-\la)$ is given by $-\left \bracevert \frac{\p \Pi_1(\la)}{\p \la} \right \bracevert ^2 _{\mathfrak{S}_2}$. But by \cite{CowenDouglas}, the curvature can be calculated via the formula $-\Delta \log \|k_{\la}\|^2$, where $k_{\la}(z)=\frac{1}{\bar{\la}z}\log \frac{1}{1-\bar{\la}{z}}$.

We can also give an alternative proof as follows:

First, the reproducing kernel property of $k_\la$ implies that $\|k_\la\|_2^2 = - \frac{\log (1-|\la|^2)}{|\la|^2}$. Therefore for $f\in \mathcal{D}$,
$$
\Pi_1(\la) f = -\frac{|\la|^2}{\log(1-|\la|^2)}f(\bar\la) k_{\bar\la}.
$$
If we take $\frac{\p}{\p\la}$ and use the fact that $\frac{\p f(\bar\la)}{\p\la}=0$, we obtain
\begin{equation}
\frac{\p\Pi_1(\la)}{\p\la} f = -\frac{f(\bar\la)\bar{\la}}{[\log (1-|\la|^2)]^2} [ \left ( \log(1-|\la|^2)+\frac{|\la|^2}{1-|\la|^2} \right ) k_{\bar\la} + \la \log(1-|\la|^2) \wt k_{\bar{\la}} ],
\end{equation}
where
$$
\wt k_{\bar\la}(z) =\frac{\p}{\p\la} k_{\bar\la}(z) = \frac{\la z^2 +z(1-\la z)\log(1-\la z)}{(1-\la z)(\la z)^2}.
$$

Next we note that for $f\in \mathcal{D}$,

\begin{equation}
\label{repr-k1}
\La f, \widetilde k_{\la}\Ra = f'(\la),
\end{equation}
to get that
$$
\|\wt k_{\la} \|_2^2 = \|\wt k_{\bar\la} \|_2^2= - \frac{(1-|\la|^2)^2 \log(1-|\la|^2)+|\la|^2-2|\la|^4}{(1-|\la|^2)^2|\la|^4}.
$$

We again use the reproducing property for $k_\la$ to show that
$$
\La\wt k_{\bar\la} , k_{\bar\la}\Ra = \frac{|\la|^2 \bar{\la}+\bar{\la}(1-|\la|^2)\log(1-|\la|^2)}{(1-|\la|^2)|\la|^4}.
$$

All of these calculations help us conclude that
$$
\|(\log(1-|\la|^2)+\frac{|\la|^2}{1-|\la|^2})k_{\bar\la} + \la \log(1-|\la|^2) \wt k_{\bar\la}\|_2^2 = \frac{[\log(1-|\la|^2)]^2+|\la|^2\log(1-|\la|^2)}{(1-|\la|^2)^2}.
$$
Thus,
$$
\left \bracevert\frac{\partial\Pi_1(\lambda)}{\partial \lambda}
\right \bracevert^2 = -\frac{\log(1-|\la|^2)+|\la|^2}{[\log(1-|\la|^2)(1-|\la|^2)]^2},
$$
and since
$\rank \frac{\p\Pi_1(\la)}{\p\la}=1$, the operator and the Hilbert--Schmidt norms of  $\frac{\p\Pi_1(\la)}{\p\la}$ are the same.
\end{proof}
Theorem 3.2 lets us rewrite statement (3) of Theorem 2.1 and equation (2.1), where the quantity
$${\left \bracevert \frac{\partial\Pi(\lambda)}{\partial\lambda}\right \bracevert}^2_{\mathfrak{S}_2} +\frac{n[\log(1-|\lambda|^2)+|\lambda|^2]}{[\log(1-|\lambda|^2)(1-|\lambda|^2)]^2}$$
is replaced by
$$
{\left \bracevert \frac{\partial\Pi_2(\lambda)}{\partial\lambda}\right \bracevert}^2_{\mathfrak{S}_2}.$$

\section{Proof of implications}

Let us begin by assuming the existence of  $\Psi$, a uniformly equivalent bundle map bijection. Since $\Psi$ is bundle map, it is an analytic function of $\la$ that takes the  fiber $\ker(D^*-\la)$ to the fiber $\ker(T-\la)$ and is linear in each fiber. We get right away that it should be of the form
$$
\Psi(k_{\bar{\lambda}}e)=k_{\bar{\lambda}} \cdot F(\la)e,
$$
for $F \in
H^{\infty}_{\C^n \rightarrow \mathcal{E}}$, an operator-valued function whose range equals $\mathcal{E}(\la)$, and all $e \in \mathbb{C}^n$.
Moreover, the ``uniform equivalence'' property of $\Psi$ says that
$$
c^{-1} I \le F^*F\le c I,
$$
for all $\la \in \mathbb{D}$. Hence the orthogonal projection
$\Pi_2(\la)$ from $\mathcal{E}$ onto $\mathcal{E}(\la)$ can be represented in terms of this function $F$:
$$
\Pi_2=F(F^*F)^{-1}F^*.
$$
Differentiation shows that $\frac{\partial\Pi_2(\la)}{\partial
\la}=(I-\Pi_2(\la))F'(\la)(F(\la)^*F(\la))^{-1}F(\la)^*$, and hence,
$$
\left \bracevert \frac{\partial\Pi_2(z)}{\partial z} \right \bracevert \leq C\nm F'(z) \nm.
$$
Now we take $\vf(\la)=\left \bracevert F(\la) \right \bracevert ^2 _{\mathfrak{S}_2}$
and note that
$$
\Delta \vf(\la)=\left \bracevert F'(\la) \right \bracevert ^2 _ {\mathfrak{S}_2}
$$
to get statement (3).

Now assuming (3), we want to prove the similarity of $T$ to the backward shift on the Dirichlet space $\mathcal{D}_{\mathcal{E}}$. The usefulness of condition (3) lies in that it makes the projection $\Pi_2(\la)$ analytic via the following Theorem by S. Treil and B. Wick \cite{TreilWick}:
\begin{thm}
If $P$ is a $\mathcal{C}^2$ function defined on $\mathbb{D}$ whose values are orthogonal projections with $P(\la)\frac{\p P(\la)}{\p \la}=0$, then the existence of a bounded subharmonic function $\vf$ satisfying
$$
\Delta \vf(\la) \geq \left \bracevert \frac{\p P(\la)}{\p \la} \right \bracevert ^2,
$$
for all $\la \in \mathbb{D}$, implies the existence of a bounded analytic function $\mathcal{P}$ whose value at $\la \in \mathbb{D}$, $\mathcal{P}(\la)$ is the projection onto $\ran P(\la)$.
\end{thm}
We proceed as in \cite{DouglasKwonTreil} and \cite{KwonTreil} and apply the theorem to $\Pi_2$. We obtain a bounded, analytic projection $\mathcal{P}(\la)$ onto $\ran \Pi_2(\la)=\mathcal{E}(\la)$ as a result and consider the inner-outer factorization of $\mathcal{P}$. We then take the inner function $\mathcal{P}_i$ and form a Toeplitz operator $T_{\mathcal{Q}}$, where the symbol $\mathcal{Q}$ is defined by
$$
\mathcal{Q}(z)=\mathcal{P}_i(\bar{z})
$$
for $z \in \mathbb{D}$. This Toeplitz operator will be the operator $A$ establishing similarity. Note that the exact same arguments used in the Hardy and the Bergman settings can be used since the multiplier algebra for the Dirichlet space is contained in that for these other spaces. For details, we refer the reader to \cite{DouglasKwonTreil} and \cite{KwonTreil}.

\def\cprime{$'$}
\providecommand{\bysame}{\leavevmode\hbox
to3em{\hrulefill}\thinspace}

\end{document}